\documentclass[12pt]{article}

\usepackage{amsmath}
\usepackage{amsthm}
\usepackage{amsfonts}
\usepackage{amssymb}

\newtheorem{theorem}{Theorem}[section]
\newtheorem{lemma}[theorem]{Lemma}

\newtheorem{proposition}[theorem]{Proposition}

\newtheorem{rem}[theorem]{Remark}

\newcommand{\od}{\overset{d}{=}}

\newcommand{\lin}{\underset{n\rightarrow\infty}{\lim}}

\newcommand{\mmp}{\mathbb{P}}
\newcommand{\me}{\mathbb{E}}
\newcommand{\mn}{\mathbb{N}}

\newcommand{\mr}{\mathbb{R}}

\newcommand{\wh}{\widehat}
\newcommand{\1}{\mathbf{1}}
\newcommand\eps{\varepsilon}
\newcommand\vth{\vartheta}

\newcommand{\mm}{\mathcal{M}}

\begin{document}

\title{Exponential rate of $L_p$- convergence of intrinsic martingales in supercritical branching random
walks\footnote{A.Iksanov, S. Polotsky and U. R\"{o}sler were
supported by the German Research Foundation (project no.\ 436UKR
113/93/0-1). The research leading to the present paper has been
mainly conducted during visits to University of Kiev
(R\"{o}sler), to University of Kiel (Iksanov and Polotsky), and to
University of M\"unster (Iksanov). Financial support obtained from
these institutions is gratefully acknowledged.}}\date{}
\author{Gerold Alsmeyer\footnote{e-mail address:
gerolda@math.uni-muenster.de}\\
\small{\emph{Institut f\"ur Mathematische Statistik},
\emph{Westf\"alische Wilhelms-Universit\"at M\"unster}},\\
\small{\emph{Einsteinstra\ss e 62, D-48149 M\"unster, Germany}}
\vspace{.3cm}\\
Alex Iksanov\footnote{Corresponding author; tel./fax.: +38044
5213202\newline e-mail address: iksan@unicyb.kiev.ua} and Sergej
Polotsky\footnote{e-mail address:
sergiy.polotskiy@gmail.com}\\
\small{\emph{Faculty of Cybernetics},
\emph{National T. Shevchenko University of Kiev}},\\
\small{\emph{01033 Kiev, Ukraine}}\vspace{.3cm}\\
Uwe R\"{o}sler\footnote{e-mail address:
roesler@math.uni-kiel.de}\\\small{\emph{Mathematisches Seminar},
\emph{Christian-Albrechts-Universit\"{a}t zu Kiel}},\\
\small{\emph{Ludewig-Meyn-Str.\ 4, D-24098 Kiel, Germany}}}

\maketitle
\begin{quote}
\noindent {\footnotesize \textbf{SUMMARY.}} Let $W_n, n\in\mn_{0}$
be an intrinsic martingale with almost sure limit $W$ in a
supercritical branching random walk. We provide criteria for the
$L_p$-convergence of the series $\sum_{n\ge 0} e^{an}(W-W_n)$ for
$p>1$ and $a>0$. The result may be viewed as a statement about the
exponential rate of convergence of $\me |W-W_n|^p$ to zero.
\medskip


\end{quote}

\section{Introduction and main results}

We start by recalling a definition of the branching random walk.
Consider~a population starting from one ancestor located at the
origin and evolving like a Galton-Watson process but with the
generalization that individuals may have infinitely many children.
All individuals are residing in points on the real line, and the
displacements of children relative to their mother are described
by a copy of a locally finite point process $\mm=\sum_{i=1}^J
\delta_{X_i}$ on $\mr$, and for different mothers these copies are
independent. Note once again that the random variable $J=\mm(\mr)$
giving the offspring number may be infinite with positive
probability. For $n\in\mn_{0}:=\{0,1,...\}$ let $\mm_n$ be the
point process that defines the positions on $\mr$ of the
individuals of the $n$-th generation. The sequence $\mm_n,
n\in\mn_{0}$ is called a \emph{branching random walk (BRW)}. In
what follows we always assume that $\me J>1$ (supercriticality)
which ensures survival of the population with positive
probability.

Every BRW is uniquely associated with a \emph{weighted branching
process (WBP)} to be formally introduced next: Let
$\mathbf{V}:=\bigcup_{n\ge 0}\mn^n$ be the infinite Ulam-Harris
tree of all finite sequences $v=v_1...v_n$ with root $\varnothing$
($\mn^0:=\{\varnothing\}$) and edges connecting each
$v\in\mathbf{V}$ with its successors $vi$, $i=1,2,...$ The length
of $v$ is denoted as $|v|$. Call $v$ an individual and $|v|$ its
generation number. Associate with every edge $(v,vi)$ of
$\mathbf{V}$ a nonnegative random variable $L_{i}(v)$ (weight) and
define recursively $L_{\varnothing}:=1$ and
$L_{vi}:=L_{i}(v)L_{v}$. The random variable $L_{v}$ may be
interpreted as the total multiplicative weight assigned to the
unique path from the root $\varnothing$ to $v$. For any
$u\in\mathbf{V}$, put similarly $L_{\varnothing}(u):=1$ and
$L_{vi}(u):=L_{i}(v)L_{v}(u)$. Then $L_{v}(u)$ gives the total
weight of the path from $u$ to $uv$. Provided that $L_{i}(v),
v\in\mathbf{V},i\in\mn$ consists of i.i.d.\ random variables, the
pair $(\mathbf{V}, \mathbf{L})$ with $\mathbf{L}:=(L_{v}(w)),
v\in\mathbf{V}, w\in \mathbf{V}$ is called a WBP with associated
BRW $\mm_n, n\in\mn_{0}$ defined as
$\mm_{n}=\sum_{|v|=n}\delta_{\log L_{v}}(\cdot\cap\mr)$. The $\log
L_{v}>-\infty$ for $v\in\mn^{n}$ are thus the positions of the
individuals alive in generation $n$. Note that, if
$u\mathbf{V}:=\{uv:v\in\mathbf{V}\}$ denotes the subtree of
$\mathbf{V}$ rooted at $u$, then the WBP on this subtree  is given
by $(u\mathbf{V},\mathbf{L}(u))$, where
$\mathbf{L}(u):=(L_{u}(v)), v\in\mathbf{V}$.

Next define
\begin{equation*}
Z_{n}\ :=\ \sum_{|v|=n}L_{v}\quad\text{and}\quad m(r)\ :=\ \me
\sum_{|v|=n}L^r_{v}
\end{equation*}
for $n\in\mn_{0}$, $r>0$ and suppose that $m(1)<\infty$. If $m$ is differentiable at $r$, then
\begin{equation}\label{xlogx}
m^\prime(r)\ =\ \me\Bigg(\sum_{|v|=1}L^r_{v}\log L_{v} \Bigg).
\end{equation}
In those cases where the right hand expectation exists but is $-\infty$ or $+\infty$ (which can only happen when $r$ is a left or right endpoint of the possibly degenerate interval
$\{r:m(r)<\infty\}$) we take \eqref{xlogx} as the definition of $m^\prime(r)$.

Let $\mathcal{F}_0$ be the trivial $\sigma$-field,
$\mathcal{F}_n:= \sigma(L_{i}(v):i\in\mn,|v|<n)$ for $n\in\mn$ and
$\mathcal{F}_{\infty}:=\sigma(\mathcal{F}_{n}:n\in\mn_{0})$. The
sequence $(W_{n},\mathcal{F}_{n}), n\in\mn_{0}$, where
\begin{equation}\label{Wn2000}
W_n\ :=\ \frac{Z_{n}}{m^{n}(1)},
\end{equation}
forms a nonnegative martingale with mean one and is thus a.s.\
convergent to a limiting variable $W$, say, satisfying $\me W \le
1$. It has been extensively studied in the literature, but first
results were obtained in \cite{King} and \cite{Big1}. Note that
$\mmp\{W>0\}>0$ if, and only if, $W_n, n\in\mn_{0}$ is uniformly
integrable. An ultimate~uniform integrability criterion was given
in \cite{AlsIks}, earlier results can be found in \cite{Big1},
\cite{Lyons}, \cite{Liu} and \cite{Iks04}.

Possibly after switching to the WBP
$(\mathbf{V},(L_{v}(w)/m^{|v|}(1), v,w \in\mathbf{V}))$ it is no
loss of generality to assume throughout that
$$ m(1)=1. $$
We further impose the condition
\begin{equation}\label{tr}
\mmp\{W_1=1\}<1
\end{equation}
which avoids the trivial situation where $\mmp\{W_n=1\}$ for
all $n\in\mn$ and hence $\mmp\{W=1\}=1$.

Other WBPs appearing in this work are the afore-mentioned
$(u\mathbf{V},\mathbf{L}(u))$ for any $u\in\mathbf{V}$ and
$(\mathbf{V},\mathbf{L}^{r})$, where
$\mathbf{L}^{r}:=(L_{v}^{r}(w)), v,w\in\mathbf{V}$. The
counterparts of $Z_{n},W_{n}$ for these processes are denoted
$Z_{n}(u),W_{n}(u)$ and $Z_{n}^{(r)},W_{n}^{(r)}$, respectively,
so $Z_{n}(u):=\sum_{|v|=n}L_{v}(u)$,
$Z_{n}^{(r)}:=\sum_{|v|=n}L_{v}^{r}$ and
$W_{n}^{(r)}:={Z_n^{(r)}\over m^n(r)}$.

The main results of this paper will provide necessary and sufficient conditions for the
$L_p$-convergence ($p>1$) of the series
\begin{equation}\label{series}
A\ :=\ \sum_{n\ge 0} e^{an}(W-W_n),
\end{equation}
for fixed $a>0$. More precisely, we will derive equivalent
necessary and sufficient conditions in the simpler case $p\ge 2$,
while a necessary and a slightly stronger sufficient condition are
presented in the surprisingly intriguing case $1<p<2$. Plainly,
our results give information on the rate of convergence of $\me
|W-W_n|^p$ to zero, as $n\to\infty$. It is therefore useful to
recall conditions (which can be found in \cite[Theorem
2.1]{Liu2000}, \cite[Corollary 5]{Iks04} or \cite[Theorem
3.1]{AlsKuhl}) ensuring that this expectation does go to $0$ or,
equivalently, that the martingale $\{W_n:n\in\mn_{0}\}$ converges
in $L_p$.

\begin{proposition}\label{lp}
Suppose (\ref{tr}) and $p>1$. Then the conditions
$$\me W_1^p<\infty \
\ \text{and} \ \ m(p)<1$$
are necessary and sufficient for
$$\lin \me |W-W_n|^p=0, $$
and the latter is equivalent to $\underset{n\geq 0}{\sup}\,\me
W_n^p<\infty$ as well as to $\me W^{p}\in (0,\infty)$.
\end{proposition}
Now we are ready to formulate our main results.

\begin{theorem}\label{main1}
Suppose (\ref{tr}), $a>0$ and $p\in (1,2)$. Then $A$ converges in
$L_p$ and almost surely if\begin{equation}\label{0} \me
W_1^{r}<\infty\quad\text{and}\quad e^a m^{1/r}(r)< 1\text{ for
some }r\in [p,2].
\end{equation}
Conversely, the $L_{p}$-convergence of $A$  implies
\begin{equation}\label{00}
\me W_1^p<\infty\quad\text{and}\quad \inf_{r\in [p,2]}e^a m^{1/r}(r)\le 1.
\end{equation}
\end{theorem}

\begin{rem}\label{supplement}\rm
In the case where the function $r\mapsto m^{1/r}(r)$ attains its
minimum at some $\theta<p$, i.e. $m(\theta)^{1/\theta}<
m^{1/p}(p)$ for some $1<\theta<p$, our analysis will actually show
that the $L_{p}$-convergence of $A$ even implies
\begin{equation*}
\me W_1^p<\infty\quad\text{and}\quad e^a m^{1/p}(p)<1,
\end{equation*}
see Remark \ref{sharper} after the proof of Theorem \ref{main1}.
Similarly, if the function $r\mapsto m^{1/r}(r)$ attains its
minimum at some $\theta\ge 2$, the $L_{p}$-convergence of $A$
implies
\begin{equation*}
\me W_1^p<\infty\quad\text{and}\quad e^a m^{1/2}(2)<1.
\end{equation*}
In other words,
$$ \inf_{r\in [p,2]}e^a m^{1/r}(r)=1 $$
in condition \eqref{00} is possible only if the last infimum is
attained at some $r \in [p,2)$.
\end{rem}

\begin{theorem}\label{main2}
Suppose (\ref{tr}), $a>0$ and $p\ge 2$. Then $A$ converges in $L_{p}$ if, and only if,
\begin{equation}\label{1}
\me W_1^p<\infty \ \ \text{and} \ \ e^a (m^{1/2}(2)\vee m^{1/p}(p)) < 1,
\end{equation}
and in this case $A$ converges also almost surely.
\end{theorem}

\begin{rem}\rm
Suppose that $A$ in \eqref{series} exists in the sense of
convergence in probability and let $A(v)$ be the corresponding
series for the subtree $v\mathbf{V}$. The $A(v),\ |v|=1$, are
independent copies of $A$ and independent of the $L_v$, $|v|=1$.
Moreover, the equation
\begin{equation}\label{fix}
A\ \od\  e^a\sum_{|v|=1}L_vA(v)+W-1
\end{equation}
holds true (in fact, even with "$=$" instead of "$\od$"). Albeit
looking like a stochastic fixed point equation it is not, for the
$A(v)$ are \emph{not} independent of the random variable $W$.
\end{rem}

\section{Size-biasing and spinal trees}

In the following, we will briefly present some required material on size-biasing
and spinal trees in connection with BRW. Generally speaking, size-biasing has
proved to be a very effective tool from harmonic analysis in the study of various
branching models. Here we restrict ourselves to a rather
informal description of those facts that are needed in this article.

Let $(\Omega,\mathcal{F},\mmp)$ denote the underlying probability
space. As $W_{n},\ n\in\mn_{0}$ constitutes a nonnegative mean one
martingale, we can uniquely define a new probability measure
$\wh\mmp$ on $(\Omega,\mathcal{F}_{\infty})$ via the projections
  $$ d\,\wh\mmp_{|\mathcal{F}_{n}}\ =\ W_{n}\,d\,\mmp_{|\mathcal{F}_{n}} $$
for all $n\in\mn_0$.

Fix $n$ and define a random variable $\Xi_{n}$ taking values in
$\mathbf{V}_n:=\{v\in\mathbf{V}:|v|=n\}$ such that
$$\wh\mmp(\Xi_{n}=v|\ \mathcal{F}_{\infty})\ =\ \frac{L_v}{W_n}. $$
Hence $\Xi_{n},n\in\mn$, picks a node in $\mathbf{V}_n$ in
accordance with the size-biased distribution obtained from
$L_{v},\ v\in\mathbf{V}_{n}$. Let $(\Xi_{0},...,\Xi_{n})$ denote
the vertices visited by the path connecting the root
$\Xi_{0}:=\varnothing$ with $\Xi_{n}$. It is not difficult to
verify that, conditioned upon $\mathcal{F}_{\infty}$, this random
vector constitutes a Markov chain on the subtree $\mathbf{V}_{\le
n}:=\{v\in\mathbf{V}:|v|\le n\}$ with one-step transition
probabilities
  $$ P(v,vi)\ :=\ {L_i(v)W_{k}(v)\over W_{k+1}(vi)},
\quad v\in\mathbf{V}_{k},\ vi\in\mathbf{V}_{k+1}. $$ Though
suppressed in the notation, it should be noticed that
$P(\cdot,\cdot)$ depends on $n$ and on $\mathcal{F}_{\infty}$. The
thus obtained random line of individuals $(\Xi_{0},...,\Xi_{n})$
in $\mathbf{V}_{\le n}$ is called its \emph{spine}, and the main
observation stated in Proposition \ref{spine} below is that these
individuals produce offspring and pick a position in a different
way than the other population members.

Define
$$ \mathcal{I}_{k}\ :=\ \big\{i\in\mn : \Xi_{k-1}i\ne \Xi_{k}\text{ and }L_{i}(\Xi_{k-1})>0\big\} $$
to be the random set of labels $i$ such that $\Xi_{k-1}i$ is
nonspinal offspring in generation $k$ of the spinal mother
$\Xi_{k-1}$. Notice that $\mathcal{I}_{k}$ may be empty. Define
further
$$ \mathcal{G}_{n}\ :=\ \sigma\Bigg(\Big(\Xi_{k},L_{\Xi_{k}},\mathcal{I}_{k}\Big)_{1\le k\le n},
\sum_{i\in\mathcal{I}_{k}}\delta_{L_i(\Xi_{k-1})}\Bigg), $$
$\mathbf{S}=\{(v,L_{v}):v\in\mathbf{V}\}$ and $\mathbf{S}_{\le
n}:=\{(v,L_{v}): |v|<n\}$. Following our usual convention, we let
$\mathbf{S}_{\le n}(v)$ denote the shifted counterpart of
$\mathbf{S}_{\le n}=\mathbf{S}_{\le n}(\varnothing)$ rooted at
$v$, more precisely
  $$ \mathbf{S}_{\le n}(v)\ :=\ \big\{(vw,L_w(v)): |w|<n \big\}. $$
The following proposition, of which parts (a)--(d) appear in a
similar form in \cite{HuShi}, provides all relevant information on
the distribution of $\mathbf{S}_{\le n}$ and the spine under
$\wh\mmp$.

\begin{proposition}\label{spine}
The following assertions hold true under the probability measure
$\wh\mmp$ for any fixed $n\in\mn$:
  \begin{itemize}
  \item[(a)] The random vectors $\big(\sum_{i\in\mathcal{I}_{k}}\delta_{L_i(\Xi_{k-1})},L_{\Xi_{k}}/L_{\Xi_{k-1}}\big)$,
  $1\le k\le n$,
are independent and identically distributed with the same
distribution as
$\big(\sum_{i\in\mathcal{I}_{1}}\delta_{L_i},L_{\Xi_{1}}\big)$.

\item[(b)] Conditioned upon $\mathcal{G}_{n}$, the shifted
weighted subtrees $\mathbf{S}_{\le n- |v|}(v)$,
$v\in\bigcup_{k=1}^{n}\mathcal{I}_{k}$, are independent, and
$\wh\mmp(\mathbf{S}_{\le n-|v|}(v)\in\cdot|\mathcal{G}_{n})\equiv
\mmp(\mathbf{S}_{n-|v|}\in\cdot)$.

\item[(c)] Putting $\Pi_{k}:=L_{\Xi_{k}}$ and
$Q_{k}:=\sum_{i\in\mn}L_i(\Xi_{k-1})$ for $k\in\mn_0$, the random
vectors $\big(\Pi_{k}/\Pi_{k-1},Q_{k},|\mathcal{I}_{k}|\big),\
1\le k\le n$, are independent copies of
$\big(\Pi_{1},Q_{1},|\mathcal{I}_{1}|\big)$. Moreover,
$\wh\me\log\Pi_{1}=m'(1)$ if $m'(1)$ exists, while
$\wh\me\log\Pi_{1}$ does not exist, otherwise. \item[(d)] For any
nonnegative measurable $f:[0,\infty)\to [0,\infty)$
\begin{equation}\label{01}
\wh\me f(\Pi_{n})\ =\ \me\Bigg(\sum_{|v|=n}L_{v}f(L_{v})\Bigg).
\end{equation}
\item[(e)] For any nondecreasing and concave function $f : [0,\infty)\to [0,\infty)$
\begin{equation}\label{02}
\wh\me f(W_{n})\ \le\ \wh\me f\left(\sum_{k=0}^{n-1}\Pi_{k}Q_{k+1}\right).
\end{equation}

\end{itemize}
\end{proposition}

We omit the proof of this result and mention only that parts (a)--(d) follow along similar arguments as those provided for supercritical Galton-Watson trees by Lyons et al. \cite{LPP}. Equality \eqref{01} may also be found in \cite{BiggKypr97}. Part (e) has been derived by Alsmeyer and Iksanov \cite{AlsIks}, see their argument to derive formula (60).

For any $\theta\ge 0$ such that $m(\theta)<\infty$, the previously
defined size-biasing can clearly be done as well with respect to
$W_{n}^{(\theta)}, n\in\mn_{0}$ by introducing the probability
measure $\wh\mmp_{\theta}$ on $\mathcal{F}_{\infty}$ defined via
the projections
$$ d\,\wh\mmp^{(\theta)}_{\mathcal{F}_{n}}\ =\ W_{n}^{(\theta)}\,d\,\mmp_{|\mathcal{F}_{n}} $$
for $n\in\mn_0$. Notice that
  \begin{equation}\label{RadonNikodym}
  \frac{d\,\wh\mmp^{(\theta)}_{\mathcal{F}_{n}}}{d\,\wh\mmp_{|\mathcal{F}_{n}}}\ =\ \frac{\Pi_{n}^{\theta-1}}{m^n(\theta)}
  \end{equation}
for each $n\in\mn_0$, because
\begin{align*}
\wh\mmp^{(\theta)}(B)\ &=\ \me\Bigg(\sum_{|v|=n}\frac{L_{v}^{\theta}}{m^n(\theta)}\1_{B}\Bigg)\\
&=\ \me\Bigg(W_{n}\sum_{|v|=n}\frac{L_{v}}{W_{n}}\frac{L_{v}^{\theta-1}}{m^n(\theta)}\1_{B}\Bigg)\\
&=\ \wh\me\Bigg(\sum_{|v|=n}\wh\mmp(\Xi_{n}=v|\mathcal{F}_{\infty})\frac{L_{v}^{\theta-1}}{m^n(\theta)}\1_{B}\Bigg)\\
&=\ m^{-n}(\theta)\,\wh\me\Big(\Pi_{n}^{\theta-1}\1_{B}\Big)
\end{align*}
for all $B\in\mathcal{F}_{n}$.

\section{Auxiliary results}

The next result will be crucial for our further analysis as explained in the subsequent Remark \ref{rem2.3}.

\begin{lemma}\label{asa}
For any fixed $a>0$, the series $A$ in (\ref{series}) converges a.s. (in $L_{p}$ for $p>1$) if, and only if, the same holds true for the series
\begin{equation}\label{ser2}
A'\ :=\ \sum_{n\ge 0} b_n(W_{n+1}-W_n),
\end{equation}
where $b_n:=\sum_{k=0}^n e^{ak}=(e^a-1)^{-1}(e^{a(n+1)}-1)$ for
$n\in\mn_0$. In this case $A'=A$ a.s.
\end{lemma}

\begin{proof}
Define $A_{m}:=\sum_{n=0}^m e^{an}(W-W_{n})$ for $m\in\mn_0$. Then
\begin{align}\label{2}
%\begin{split}
A_{m}\ &=\ \lim_{l\to\infty}\sum_{n=0}^m e^{an}
\sum_{k=n}^{l} (W_{k+1}-W_k)\notag\\
&=\ \lim_{l\to\infty}\sum_{k=0}^{l}(W_{k+1}-W_k)\sum_{n=0}^{k\wedge m}e^{an}\notag\\ \noalign{\break}
&=\ \sum_{k=0}^{\infty}b_{k\wedge m}(W_{k+1}-W_{k})\notag\\
&=\ b_{m}(W-W_{m})\ +\ A_{m-1}'\quad\text{a.s.}
%\end{split}
\end{align}
where $A_{m-1}':=\sum_{k=0}^{m-1}b_{k}(W_{k+1}-W_{k})$. Now, if $A$ in \eqref{series} converges a.s., then $\lim_{m\to\infty}b_{m}(W-W_{m})=0$ a.s.\ and thus, by letting $m$ tend to infinity in \eqref{2}, we see that
$A'$ converges a.s.\ and equals $A$. Conversely, given the almost sure~convergence of $A'$, a tail sum analogue of Kronecker's lemma (see \cite[Lemma 4.2]{Asm}) ensures that $\lim_{n\to\infty} e^{an}(W-W_n)=0$ a.s. This in turn allows us to read \eqref{2} backwards thus concluding the a.s.\ convergence of $A$ as well as $A=A'$ a.s.

If $A$ is $L_{p}$-convergent for some $p>1$, then $||A_m-A||_p\to
0$ and therefore $e^{am}||W-W_m||_p=||A_{m+1}-A_m||_p \to 0$ as
$m\to\infty$. Now use \eqref{2} to infer with the help of
Minkowski's inequality
$$ \|A_{m-1}^\prime\|_{p}\ \le\ b_m\|W-W_{m}\|_{p}\ +\ \|A_{m}\|_{p} $$
and thereupon the $L_{p}$-boundedness of the martingale $A_{n}',
n\in\mn_{0}$. Consequently (see, for example, \cite[Proposition
IV-2-7]{Neveu} and its proof), $A'$ defined in \eqref{ser2}
converges a.s.\ as well as in $L_{p}$. Conversely, if $A'$ is
$L_{p}$-convergent, then by an appeal to Burkholder's inequality
(see Lemma \ref{bu} below)
\begin{align*}
b_{m}^{p}\,\me|W-W_{m}|^{p}\ &\le\ Cb_{m}^{p}\,\me\left(\sum_{n\ge m}(W_{n+1}-W_{n})^{2}\right)^{p/2}\\
&\le\ C\,\me\left(\sum_{n\ge m}b_{n}^{2}(W_{n+1}-W_{n})^{2}\right)^{p/2}\\
&\le\ C\,\me|A'-A_{m-1}'|^{p}\ \to\ 0\quad\text{as } m\to\infty,
\end{align*}
where $C\in (0,\infty)$ is a generic constant that may differ from line to line. With this result we infer from \eqref{2}
\begin{align*}
\|A_{m+n}&-A_{m}\|_{p}\\
&\le\ b_{m}\|W-W_{m}\|_{p}+b_{m+n}\|W-W_{m+n}\|_{p}
+\|A_{m+n-1}'-A_{m-1}'\|_{p}\\
&\le\ 2\sup_{k\ge m}b_{k}\|W-W_{k}\|_{p}
+\|A_{m+n-1}'-A_{m-1}'\|_{p}\\
&\to\ 0\quad\text{as } m,n\to\infty
\end{align*}
and thus the asserted $L_{p}$-convergence of $A$.
\end{proof}

\begin{rem}\label{rem2.3}\rm
(a) As, for each $n\in\mn_0$,
\begin{equation*}
A_{n}'\ =\ \frac{e^{a}}{e^{a}-1}\sum_{k=0}^{n}e^{ak}(W_{k+1}-W_{k})\ -\
\frac{1}{e^{a}-1}(W_{n+1}-1)
\end{equation*}
the proof of Lemma \ref{asa} may easily be extended to show further that $A$ converges a.s.\ (or in $L_{p}$ for $p>1$) if, and only if, this holds true for
\begin{equation}\label{ser3}
\wh{A}:=\sum_{n\ge 0}e^{an}(W_{n+1}-W_{n}).
\end{equation}
In this case, $\wh A$ is readily seen to satisfy\begin{equation}\label{fix2}
\wh{A}\ \od\ e^{a}\sum_{|v|=1}L_{v}\wh{A}_{v}+W_{1}-1
\end{equation}
with $\wh{A}_{v}$ being independent copies of $\wh{A}$ which are also independent of $W_{1}$. Hence, unlike \eqref{fix} for $A$, \eqref{fix2} constitutes a proper stochastic fixed point equation.

\vspace{.2cm}\noindent
(b) The motivation behind dealing with $\wh A$ in (\ref{ser3}) hereafter rather than $A$ in (\ref{series}) stems from the fact that the partial sums $\wh A_{n}:=\sum_{k=0}^{n}e^{ak}(W_{k+1}-W_{k})$, $n\in\mn_{0}$, constitute a martingale whereas those associated with $A$ do not. This entails that $\wh A$ forms a martingale limit (like $A'$) and as such is easier to deal with. Indeed, as far as the $L_p$-convergence ($p>1$) is concerned, a well-known property of martingales (already used in the previous proof) tells us that it suffices to prove $\me |\wh A|^p<\infty$ or, equivalently, $L_{p}$-boundedness of the $\wh A_{n}$ (see \cite[Proposition IV-2-7]{Neveu}).
\end{rem}

The proof of Theorem \ref{main1} hinges to a large extent on Proposition
\ref{limitchar} on the functions $s_{n}(r)$ defined below. The connection is provided by an application of Burkholder's inequality
which in turn is stated for reference as Lemma \ref{bu} at the end of this section.

\begin{lemma}\label{squarefunction}
Let $1<p<2$ and $W_{n}, n\in\mn_{0}$ be uniformly integrable with
$\me W_{1}^{p}<\infty$. Then the function
$$ [1,2]\ \ni\ r\ \mapsto\ s_{n}(r)\ :=\ \me\big(Z_{n}^{(r)}\big)^{p/r}\ =\ \me\Bigg(\sum_{|v|=n}L_{v}^{r}\Bigg)^{p/r} $$
is decreasing and bounded by $\sup_{n\ge 0}\me W_{n}^{p}$ for each
$n\in\mn$. Furthermore,
\begin{equation}\label{supmulti}
s_{n}(r)\
\begin{cases}
\le\ s_{k}(r)s_{n-k}(r),&\quad\text{if }r\in [1,p],\\
\ge\ s_{k}(r)s_{n-k}(r),&\quad\text{if }r\in [p,2]
\end{cases}
\end{equation}
for $0\le k\le n$, and
\begin{equation}\label{limit}
\lim_{n\to\infty}s_{n}^{1/n}(r)\
\begin{cases}
=\ \inf_{j\ge 1}s_{j}^{1/j}(r),&\quad\text{if }r\in [1,p],\\
=\ \sup_{j\ge 1}s_{j}^{1/j}(r),&\quad\text{if }r\in [p,2].
\end{cases}
\end{equation}
\end{lemma}

\begin{proof}
The first assertion follows immediately from $s_{n}(1)=\me W_{n}^{p}$ and
\begin{equation*}
\me\Bigg(\sum_{|v|=n}L_{v}^{r}\Bigg)^{p/r}\ =\ \me\Bigg(\sum_{|v|=n}L_{v}^{q\cdot (r/q)}\Bigg)^{p/r}
\ <\ \me\Bigg(\sum_{|v|=n}L_{v}^{q}\Bigg)^{p/q}
\end{equation*}
for any $1\le q<r\le 2$, where supercritical branching and strict superadditivity
of $x\mapsto x^{r/q}$ have been utilized. As for \eqref{supmulti}, we obtain in the case $r\in [p,2]$ with the help of Jensen's inequality
\begin{align*}
s_{n}(r)\ &=\ \me\Bigg(\sum_{|v|=k}L_{v}^{r}Z_{n-k}^{(r)}(v)\Bigg)^{p/r}\\
&=\ \me\left(\big(Z_{k}^{(r)}\big)^{p/r}\Bigg(\sum_{|v|=k}\frac{L_{v}^{r}}{Z_{k}^{(r)}}Z_{n-k}^{(r)}(v)\Bigg)^{p/r}\right)\\ %\noalign{\break}
&\ge\ \me\left(\big(Z_{k}^{(r)}\big)^{p/r}\sum_{|v|=k}\frac{L_{v}^{r}}{Z_{k}^{(r)}}\big(Z_{n-k}^{(r)}(v)\big)^{p/r}\right)\\
&=\ \me\left(\big(Z_{k}^{(r)}\big)^{p/r}\sum_{|v|=k}\frac{L_{v}^{r}}{Z_{k}^{(r)}}\me\Big(\big(Z_{n-k}^{(r)}(v)\big)^{p/r}\Big|\mathcal{F}_{k}\Big)\right)\\
&=\ \me\left(\big(Z_{k}^{(r)}\big)^{p/r}\sum_{|v|=k}\frac{L_{v}^{r}}{Z_{k}^{(r)}}s_{n-k}(r)\right)\\
&=\ s_{k}(r)s_{n-k}(r)
\end{align*}
for all $0\le k\le n$, and this further yields, by superadditivity
of $\log s_{n}(r)$, that $s_{n}(r)^{1/n}$ converges as
$n\to\infty$ with limit satisfying \eqref{limit}. If $r\in [1,p]$
and thus $x\mapsto x^{p/r}$ is convex, the above estimation holds
with reverse inequality sign.
\end{proof}

Notice that $\log s_{n}(r)$ is always a superadditive or
subadditive function but may be infinite. Precise information on
the asymptotic value of $s_{n}^{1/n}(r)$ as $n\to\infty$ is
provided by the next lemma. Put $g(r):=r^{-1}\log m(r)$ with
derivative
\begin{equation}\label{derivative}
g'(r)\ =\ \frac{h(r)}{r^{2}}\quad\text{with } h(r)\ :=\ \frac{rm'(r)}{m(r)}-\log m(r)
\end{equation}
on the interior of $\Bbb{D}:=\{r:m(r)<\infty\}$. Note that
$[1,p]\subset\Bbb{D}$ if $W_{n}, n\in\mn_{0}$ is uniformly
integrable and $\me W_{1}^{p}<\infty$. By supercriticality, the
function $m$ is strictly logconvex which in turn implies that $h$
is increasing with at most one zero. Therefore, the function $g$
possesses at most one minimum. Put
$$ \vth\ :=\ 2\wedge \arg\inf_{r\ge 1}\,g(r) \ \ \text{and} \ \ \gamma:=m^{1/\vth}(\vth).$$
If $\vth\in\text{int}(\Bbb{D})$ and thus $m$ is differentiable at
$\vth$, then $g'(\vth)=0$ may be rewritten as
\begin{equation}\label{minimum}
\frac{m'(\vth)}{m(\vth)}\ =\ \frac{1}{\vth}\log m(\vth).
\end{equation}
Let us also point out that $m(r)<1$ and $m'(r)<0$ for all $r\in
(1,\vth)$ because $g(r)$ has negative (right) derivative $m'(1)$
at 1 as a consequence of uniform integrability of $W_{n},
n\in\mn_{0}$.

\begin{proposition}\label{limitchar}
Suppose the assumptions of Lemma \ref{squarefunction} be true and furthermore $m(p)<1$. Let $\vth,\gamma$ be as defined above. Then, if $p\le\vth$,
\begin{align*}
\lim_{n\to\infty}s_{n}^{1/n}(r)\ =\
\begin{cases}
m^{p/r}(r),\text{ if }r\in [1,\vth)\\
\gamma^{p},\text{ if }r\in [\vth,2],
\end{cases}
\end{align*}
while, if $p>\vth$,
\begin{align*}
\lim_{n\to\infty}s_{n}^{1/n}(r)\ =\
\begin{cases}
m^{p/r}(r),\text{ if }r\in [1,q)\\
m(p),\text{ if }r\in [q,2],
\end{cases}
\end{align*}
where $q$ is the unique value in $(1,\vth)$ such that $g(q)=g(p)$, i.e.\ $m^{1/q}(q)=m^{1/p}(p)$.
\end{proposition}

Notice that in both cases above the obtained limit function $s_{\infty}(r)$, say, is continuous at its "critical" value $\vth$, respectively $q$. Also, this limit function for $p>\vth$ converges to the one for $p=\vth$, for then $q$ equals $\vth$ as well.

\begin{proof} {\sc Case A}. $p\le\vth$ and $r\in [\vth,2]$. \emph{Lower estimate}

\vspace{.2cm}\noindent
Since $s_{n}(r)$ is decreasing in $r$ it suffices to show
\begin{equation}\label{23}
\liminf_{n\to\infty}s^{1/n}_{n}(2)\ \ge\ \gamma^{p}.
\end{equation}

\vspace{.2cm}\noindent
{\sc Subcase A.1}. $\gamma=m^{1/\vth}(\vth)$ for $\vth\in (1,2)$.

\vspace{.1cm}\noindent
An old result by Biggins \cite{Big0},\cite{Big2} tells us that
$$ \frac{\log M_{n}}{n}\ \to\ \frac{1}{\vth}\log m(\vth)\quad\text{a.s. on }\{W>0\}, $$
where $M_{n}:=\max_{|v|=n}L_{v}$. By using this fact in combination with the obvious inequality
$$ \big(Z_{n}^{(2)}\big)^{p/2}\ \ge\ M_{n}^{p}\quad\text{on }\{W>0\}, $$
we infer with the help of Jensen's inequality and Fatou's lemma
\begin{align*}
\liminf_{n\to\infty}s_{n}^{1/n}(2)\ &\ge\ \liminf_{n\to\infty}\me^{1/n}\big(M_{n}^{p}\1_{\{W>0\}}\big)\\
&=\ \liminf_{n\to\infty}\mmp^{1/n}\{W>0\}\me^{1/n}\big(M^p_n\,\big|W>0\big)\\
&\ge\ \liminf_{n\to\infty}\me\left(M_{n}^{p/n}|W>0\right)\\
&=\ m^{p/\vth}(\vth)\ =\ \gamma^{p}.
\end{align*}

\vspace{.2cm}\noindent {\sc Subcase A.2}. $\gamma=m^{1/2}(2)$
(thus $\vth=2$) and $W_{1}$ is a.s.\ bounded.

\vspace{.1cm}\noindent Then $m(2)<1$ and $m'(2)<0$ as pointed out
after \eqref{minimum}. Moreover, the almost sure boundedness
$W_{1}$ trivially ensures the same for $W_{1}^{(2)}$, in
particular $\me W_{1}^{(2)}\log^{+}W_{1}^{(2)}<\infty$. Therefore
the mean one martingale $W_{n}^{(2)}, n\in\mn_{0}$ is uniformly
integrable (cf.\ e.g.\ \cite[Theorem 1.3]{AlsIks}) and hence
convergent a.s.\ and in $L_{1}$ to a random variable $W^{(2)}$.
Since $p/2<1$, it follows that
$\me\big(W_{n}^{(2)}\big)^{p/2}\to\me\big(W^{(2)}\big)^{p/2}$ and
therefore
\begin{equation*}
s_{n}^{1/n}(2)\ =\ m^{p/2}(2)\,\me^{1/n}\big(W_{n}^{(2)}\big)^{p/2}\ \to\ m^{p/2}(2)\ =\ \gamma^{p},
\end{equation*}
as $n\to\infty$. Notice that we have indeed verified the stronger assertion that
\begin{equation}\label{30}
\lim_{n\to\infty}\frac{s_{n}(2)}{m^{pn/2}(2)}\ =\ \me\big(W^{(2)}\big)^{p/2}.
\end{equation}

\vspace{.2cm}\noindent
{\sc Subcase A.3}. $\gamma=m^{1/2}(2)$, general situation.

\vspace{.1cm}\noindent Here we use a truncation argument. For a
constant $K>0$ consider the WBP $(\mathbf{V},(\overline{L}_{v}(w),
v,w\in\mathbf{V}))$ with
\begin{equation}\label{truncate}
\overline{L}_{i}\ :=\ L_{i}\1_{\{L_{i}\ge 1/K\,,\,W_{1}(v)\le
K\}},\quad i\in\mn, v\in\mathbf{V}.
\end{equation}
This provides us with a thinning of the original WBP such that $\overline{m}(\theta):=\me(\sum_{i\ge 1}\overline{L}_{i}^{\theta})$ satisfies
$$ \overline{m}(\theta)<\infty\quad\text{and}\quad \overline{m}(\theta)\le m(\theta) $$
for all $\theta>0$. Moreover, in the obvious notation,
$$ \overline{s}_{n}(\theta)\ \le\ s_{n}(\theta) $$
for all $\theta\in [1,2]$. Plainly, as $K\to\infty$, $\overline{m}$ converges to $m$ uniformly on compact subsets contained in the interior of $\Bbb{D}$. Hence, by choosing $K$ large enough, we have for the obviously defined $\overline{\gamma}$ that
$$ \overline{\gamma}\ \ge\ (1-\eps)\gamma $$
for any fixed $\eps\in (0,1)$. By applying the result obtained
under Subcase A.2 to the normalized WBP
$(\mathbf{V},(\overline{L}_{v}(w)/\overline{m}^{|v|}(1), v,w
\in\mathbf{V}))$ we now arrive at the desired conclusion here as
well.

\vspace{.2cm}\noindent
{\sc Case A}. $p\le\vth$ and $r\in [\vth,2]$. \emph{Upper estimate}

\vspace{.1cm}\noindent
The next step is to verify
\begin{equation}\label{22}
\limsup_{n\to\infty}s_{n}^{1/n}(r)\ \le\ \gamma^{p}
\end{equation}
for each $r\in [\vth,2]$ which, in combination with $s_{n}(2)\le
s_{n}(r)$ and \eqref{23}, clearly gives the assertion of the lemma
for $r\in [\vth,2]$ and $p\le\vth$.

\vspace{0.2cm}\noindent Suppose first that $p<\vth$. Fix any
$\eps>0$ and $\theta\in (p,\vth)$ such that
$m^{1/\theta}(\theta)\le (1+\eps)\gamma<1$. Then, by another use
of Jensen's inequality,
\begin{align*}
\limsup_{n\to\infty}s_{n}^{1/n}(r)\ &\le\ \limsup_{n\to\infty}s_{n}^{1/n}(\theta)\\
&=\ \limsup_{n\to\infty}\left(m^{pn/\theta}(\theta)\,\me\big(W_{n}^{(\theta)}\big)^{p/\theta}\right)^{1/n}\\ %\noalign{\break}
&\le\ m^{p/\theta}(\theta)\,\limsup_{n\to\infty}\me^{p/n\theta}W_{n}^{(\theta)}\\
&=\ m^{p/\theta}(\theta)\\
&\le\ (1+\eps)^{p}\gamma^{p}
\end{align*}
which shows \eqref{22} as $\eps>0$ was picked arbitrarily. Now, if $p=\vth$,
we arrive at the same conclusion by choosing $\theta=p$ and $\eps=0$ in the above estimation.

\vspace{.3cm}\noindent
{\sc Case B}. $p\le\vth$ and $r\in [1,\vth)$. \emph{Lower estimate}

\vspace{.1cm}\noindent
Here we must verify
\begin{equation}\label{24}
\liminf_{n\to\infty}s_{n}^{1/n}(r)\ \ge\ m^{p/r}(r).
\end{equation}
In view of the truncation \eqref{truncate} described under Subcase A.3 it is no loss of generality to assume directly that $W_{1}$ (and thus $W_{1}^{(r)}$ as well) is a.s.\ bounded and $m(\theta)<\infty$ for all $\theta>0$. Write
\begin{equation}\label{25}
s_{n}(r)\ =\ m^{pn/r}(r)\,\me\big(W_{n}^{(r)}\big)^{p/r}
\end{equation}
for $n\in\mn_0$ and consider the WBP
$(\mathbf{V},\mathbf{L}^{r})$. Since $\theta\mapsto
m^{-\theta}(r)\,\me(\sum_{|v|=1}L_{v}^{r\theta})=m(r\theta)/m^{\theta}(r)$
has derivative
$$ \frac{rm'(r\theta)}{m^{\theta}(r)}-\log m(r)\frac{m(r\theta)}{m^{\theta}(r)} $$
taking value $r\big(\frac{m'(r)}{m(r)}-\log
m^{1/r}(r)\big)=r^{2}g'(r)<0$ at $\theta=1$, we infer (see
\cite[Theorem 1.3]{AlsIks}) that $W_{n}^{(r)}$ converges a.s.\ and
in $L_{1}$ to the random variable $W^{(r)}$ which in turn entails
\eqref{24} because, by \eqref{25} and an appeal to Jensen's
inequality and Fatou's lemma,
\begin{align*}
\liminf_{n\to\infty}s_{n}^{1/n}(r)\ &=\
m^{p/r}(r)\liminf_{n\to\infty}\me^{1/n}\big(W_{n}^{(r)}\big)^{p/r}\\
&=\ m^{p/r}(r)\liminf_{n\to\infty}\mmp^{1/n}(W^{(r)}>0)\,\me^{1/n}\Big(\big(W_{n}^{(r)}\big)^{p/r}\Big|W^{(r)}>0\Big)\\
&\ge\ m^{p/r}(r)\,\me\left(\liminf_{n\to\infty}\big(W_{n}^{(r)}\big)^{p/rn}\Big|W^{(r)}>0\right)\\
&=\ m^{p/r}(r).
\end{align*}

\vspace{.3cm}\noindent
{\sc Case B}. $p\le\vth$ and $r\in [1,\vth)$. \emph{Upper estimate}

\vspace{.1cm}\noindent
The converse
\begin{equation}\label{26}
\limsup_{n\to\infty}s_{n}^{1/n}(r)\ \le\ m^{p/r}(r)
\end{equation}
follows quite easily from \eqref{25}, for $\me\big(W_{n}^{(r)}\big)^{p/r}\le\me^{p/r}W_{n}^{(r)}=1$ for each $n\in\mn_0$ in the case $r\in [p,\vth]$ by Jensen's inequality, while in the case $r\in [1,p)$ we have $\sup_{n\ge 0}\me
\big(W_{n}^{(r)}\big)^{p/r}<\infty$ as a consequence of $\me\big(W_{1}^{(r)}\big)^{p/r}\le m^{-p/r}(r)\,\me W_{1}^{p}<\infty$ and
$$ E\Bigg(\sum_{|v|=1}\bigg(\frac{L_v}{m(r)}W_1^{(r)}\bigg)^{p/r}\Bigg)\ =\ \frac{m(p)}{m^{p/r}(r)}\ =\ e^{p(g(p)-g(r))}\ <\ 1 $$
(apply Proposition \ref{lp} to $W_{n}^{(r)}, n\in\mn_{0}$).

\vspace{.3cm}\noindent
{\sc Case C}. $p>\vth$ and $r\in[1,q).$  \emph{Upper estimate}.

\vspace{.1cm}\noindent
Notice that $m(\vth)<\infty$. As $g(\vth)<g(p)<0=g(1)$, there exists a unique $1<q<\vth$ such that $g(q)=g(p)$, i.e.\ $m(q)^{1/q}=m(p)^{1/p}$. Then, for $r\in [1,q)$, the previously given arguments are easily seen
to carry over to the present situation thus showing \eqref{26}.

\vspace{.2cm}\noindent
{\sc Case C}. $p>\vth$ and $r\in[1,q)$. \emph{Lower estimate}.

\vspace{.1cm}\noindent
By Jensen's inequality,
\begin{equation*}
s_{n}(r)\ \ge\ \me^{p/q}\big(Z_{n}^{(r)}\big)^{q/r}
\end{equation*}
But $\me\big(Z_{n}^{(r)}\big)^{q/r}$, call it $\widetilde{s}_{n}(r)$, is just the counterpart of $s_{n}(r)$ for $q<\vth$ instead of $p$. Therefore $\widetilde{s}_{n}^{\,1/n}(r)\to m^{q/r}(r)$ by what has been shown under Case B.
It thus follows that $s_{n}^{1/n}(r)$ has also the required lower bound
which completes the proof of
$$ \lim_{n\to\infty}s_{n}^{1/n}(r)\ =\ m^{p/r}(r) $$
for all $r\in [1,q)$.

\vspace{.3cm}\noindent
{\sc Case D}. $p>\vth$ and $r\in[q,2]$. \emph{Upper estimate}.

\vspace{.1cm}\noindent Since
$s_{n}(q)=\inf_{\theta<q}s_{n}(\theta)$, we obtain as a
consequence of Case B that
$$ \limsup_{n\to\infty}s_{n}^{1/n}(r)\ \le\ \limsup_{n\to\infty}s_{n}^{1/n}(q)\ \le\
 \inf_{\theta\in [1,q)}m^{p/\theta}(\theta)\ =\ m^{p/q}(q)\ =\ m(p). $$

 \vspace{.2cm}\noindent
{\sc Case D}. $p>\vth$ and $r\in[q,2]$. \emph{Lower estimate}.

\vspace{.1cm}\noindent
The proof for Case C will now be completed by showing that
\begin{equation}\label{27}
\liminf_{n\to\infty}s_{n}^{1/n}(2)\ \ge\ m(p)
 \end{equation}
(since $s_n(r)$ is decreasing in $r$) which is the most delicate part of the whole proof. Once again, possibly after a suitable truncation as described in \eqref{truncate}, it is no loss of generality to assume that $W_{1}\le K$ for some $K\ge 1$, $J\le N$ for some $N\in\mn$ and $m(\theta)<\infty$ for all  $\theta>0$. Notice also that, by subadditivity of $x\mapsto x^{p/2}$, we find
\begin{equation}\label{28}
s_{n}(2)\ \le\ m^n(p)
\end{equation}
for all $n\in\mn_0$.

\vspace{.2cm}\noindent
Put $\beta:=1-(p/2)\in (0,1)$. Recall the notation introduced in Section 2 in connection with the size-biased probability measure $\wh\mmp$. We have
  \begin{align*}
  (Z_{n}^{(2)})^{p/2}\ &=\   \Bigg(\Pi_{n}^{2}+\sum_{k=1}^{n}
    \sum_{i\in\mathcal{I}_{k}}L_v^2Z_{n-k}^{(2)}(\Xi_{k-1}i)\Bigg)^{p/2}\\
  &\le\ K^{p}\Bigg(\Pi_{n}^{p}+\sum_{k=1}^{n}\Pi_{k-1}^{p}
  (\Lambda_{n,k}^{(2)})^{p/2}\Bigg)\quad\wh\mmp\text{-a.s.},
  \end{align*}
where, for $1\le k\le n$,
  $$ \Lambda_{n,k}^{(2)}\ :=\ \sum_{i\in\mathcal{I}_{k}}Z_{n-k}^{(2)}(\Xi_{k-1}i). $$
Use Proposition \ref{spine}(b), to see that conditioned upon
$\mathcal{G}_{n}$, the $Z_{n-k}^{(2)}(\Xi_{k-1}i),
i\in\mathcal{I}_{k}$, are i.i.d.\ under $\wh\mmp$ with the same
distribution as $Z_{n-k}^{(2)}$ under $\mmp$. By combining this
with another subadditivity argument we obtain $\wh\mmp$-a.s.
\begin{equation*}
\wh\me\big((\Lambda_{n,k}^{(2)})^{p/2}|\mathcal{G}_{n}\big)\ \le\
\wh\me\left(\sum_{i\in\mathcal{I}_{1}}(Z_{n-k}^{(2)}(\Xi_{k-1}i))^{p/2}\Bigg|\mathcal{G}_{n}\right)  \
\le\ |\mathcal{I}_{1}|\,\me (Z_{n-k}^{(2)})^{p/2}
\end{equation*}
for $k=1,\ldots,n$. As $|\mathcal{I}_{1}|\le N$ for some $N\in\mn$
by truncation we arrive at
\begin{align*}
\wh\me\big((Z_{n}^{(2)})^{p/2}|\mathcal{G}_{n}\big)\ &\le\ K^{p}N\Bigg(\Pi_{n}^{p}+\sum_{k=1}^{n}\Pi_{k-1}^{p}\,\me (Z_{n-k}^{(2)})^{p/2}\Bigg)\\
&=\ K^{p}N\sum_{k=0}^{n}\Pi_{k}^{p}s_{n-k}(2)\quad\wh\mmp\text{-a.s.}
\end{align*}
Now use $\me (Z_{n}^{(2)})^{p/2}=\wh\me\big(\Pi_{n}(Z_{n}^{(2)})^{-\beta}\big)$, which follows from
\begin{align*}
\me (Z_{n}^{(2)})^{p/2}\ &=\   \wh\me\Bigg(\sum_{|v|=n}
\frac{L_{v}^{2}}{W_{n}}(Z_{n}^{(2)})^{-\beta}\Bigg)\\
&=\ \wh\me\Bigg(\sum_{|v|=n}\wh\mmp(\Xi_{n}=v|\mathcal{F}_{\infty})
L_{v}(Z_{n}^{(2)})^{-\beta}\Bigg)\\
&=\ \wh\me\Bigg(\sum_{|v|=n}\1_{\{\Xi_{n}=v\}}L_{v}(Z_{n}^{(2)})^{-\beta}\Bigg)\ =\ \wh\me\Big(\Pi_{n}(Z_{n}^{(2)})^{-\beta}\Big),
\end{align*}
to obtain by an appeal to Jensen's inequality for $x\mapsto x^{(p-2)/p}$
\begin{align*}
\me (Z_{n}^{(2)})^{p/2}\ &=\   \wh\me\Big(\Pi_{n}\,\wh\me\big((Z_{n}^{(2)})^{-\beta}|\mathcal{G}_{n}\big)\Big)\\
&\ge\ \wh\me\Bigg(\frac{\Pi_n}{\wh\me\big((Z_{n}^{(2)})^{p/2}|\mathcal{G}_{n}\big)^{(2-p)/p}}\Bigg)\\
&\ge\
C\,\wh\me\left(\frac{\Pi_n}{\big(\sum_{k=0}^{n}\Pi_k^ps_{n-k}(2)\big)^{(2-p)/p}}\right)
\end{align*}
for some $C>0$. Recall from Section 2 the definition of $\wh\mmp_{p}$ and that (see \eqref{RadonNikodym})
$$ \wh\mmp^{(p)}(B)\ =\ \frac{1}{m^{n}(p)}\,\wh\me\Big(\Pi_n^{p-1}\1_{B}\Big) $$
for any $B\in\mathcal{F}_{n}$. The last expectation can be further estimated as
\begin{align*}
\wh\me&\left(\frac{\Pi_n}{\big(\sum_{k=0}^{n}s_{n-k}(2)\Pi_k^p\big)^{(2-p)/p}}\right)\\ %\noalign{\break}
&=\ m^{n}(p)\,\wh\me^{(p)}\left(\frac{\Pi_n^{2-p}}{\big(\sum_{k=0}^{n}s_{n-k}(2)\Pi_n^k\big)^{(2-p)/p}}\right)\\
&=\ m^{n}(p)\,\wh\me^{(p)}\left(\frac{1}{\big(\sum_{k=0}^{n}s_{n-k}(2)
(\Pi_n/\Pi_k)^p\big)^{(2-p)/p}}\right)\\ %\noalign{\break}
&=\  m^{n}(p)\,\wh\me^{(p)}\left(\frac{1}{\big(\sum_{k=0}^{n}s_{k}(2)\Pi_k^p\big)^{(2-p)/p}}\right)\\ %\noalign{\break}
&\ge\ m^{n}(p)\,\wh\me^{(p)}\left(\frac{1}{\big(\sum_{k=0}^{n}
m^{k}(p) \Pi_k^p\big)^{(2-p)/p}}\right)\\
&\ge\ m^{n}(p)\,\wh\me^{(p)}\left(\frac{1}{\big(
\sum_{k\ge  0}\big(\Pi_k^*\big)^p  \big)^{(2-p)/p}}\right)\\
\end{align*}
where \eqref{28} has been utilized for the penultimate inequality and where
$\Pi_k^*:= \Pi_{k}/m^{k/p}(p)$ for $k\in\mn_{0}$. Since $\wh\me^{(p)} \Pi_1^\theta =\frac{m(p+\theta)}{m(p)}$ for all $\theta\in\mr$ we find
$$ \wh\me^{(p)}\log \Pi_1^* \ =\ \frac{m'(p)}{m(p)}-\frac{1}{p}\log m(p)\ =\ \frac{h(p)}{p}. $$
Now use $p>\vth$ to infer $\wh\me^{(p)}\ln \Pi_1^*>0$ and thereupon that
$$ 1\ <\ \Sigma\ :=\ \sum_{k\ge 0}(\Pi_k^*)^{-p}\ <\ \infty\quad\wh\mmp_{p}\text{-a.s.}, $$
in particular $\nu:=\wh\me^{(p)}\Sigma^{-(2-p)/p}\in (0,1)$.
We finally arrive at
\begin{equation}\label{29}
m^{n}(p)\ \ge\ s_{n}(2)\ =\ \me Z_{n}^{p/2}\ \ge\ \nu\, m^{n}(p)
\end{equation}
for all $n\in\mn_{0}$ which clearly implies the desired assertion
\eqref{27}. The proof of Proposition \ref{limitchar} is thus
complete.
\end{proof}

The next lemma is needed for the proof of Theorem \ref{main2} and examines the asymptotic behaviour of $\me W_{n}^{p}$ when $\me W_{1}^{p}<\infty$ and $m(p)\geq 1$. It may also be viewed as a useful complement to Proposition \ref{lp}. Let us
mention that we have not tried to obtain the best possible estimates. Actually, for our purposes only factors of exponential growth matter. Hence, we content ourselves with quite crude estimates when dealing with factors of subexponential growth.

\begin{lemma}\label{ma}
Let $p>1$ and $\me W_1^p<\infty$.
\begin{itemize}
\item[(a)] If $p\in (1,2]$, then $\me W_n^p=O(n)$ if $m(p)=1$ and $\me W_n^p=O(m^{n}(p))$ if $m(p)>1$.
\item[(b)] If $p>2$, then $\me W_n^p=O(n^{p-1})$ if $m(p)=1$ and $\me W_n^p=O(n^{b(p-1)}m^{n}(p))$ for $p\in
(b+1,b+2]$, $b\in\mn$, if $m(p)>1$.
\end{itemize}
\end{lemma}

We note in advance that the lemma will later be applied to the
martingale $\{W_{n}^{(2)}:n\in\mn_{0}\}$ rather than to $\{W_n:n\in\mn_{0}\}$.

\begin{proof}
(a) Use Proposition \ref{spine}(e) with $f(x)=x^{p-1}$ to infer
\begin{align*}
\me W_n^p\ &\le\ \wh\me \left(\sum_{k=0}^{n-1}M_1 M_2\cdots M_k
Q_{k+1}\right)^{p-1}\\
&\le\ \wh\me \left(\sum_{k=0}^{n-1}(M_1M_2\cdots
M_k)^{p-1} Q_{k+1}^{p-1}\right)\ =\ \wh\me Q^{p-1}\sum_{k=0}^{n-1}\wh\me ^{k}M^{p-1},
\end{align*}
where in the next to last inequality the subadditivity of $f$ has
been utilized and $(M,Q)$ denotes a generic copy of the
$(M_{n},Q_{n})$. In view of Proposition \ref{spine}, $\wh \me
Q^{p-1}=\me W_1^p$ and $\wh \me M^{p-1}=m(p)$, and the result
follows.

\vspace{.2cm}\noindent
(b) Put $\varphi_n(s):=\me e^{-sW_n}$ for $n\in\mn_{0}$ Then
$$ \varphi_n(s)=\me \prod_{|v|=1}\varphi_{n-1}(sL_v),\quad n\in\mn.$$
Differentiating this equality yields
\begin{equation}\label{ik}
\varphi_n^\prime(s)\ =\ \sum_{|v|=1}\varphi_{n-1}^\prime(sL_v)L_v\prod_{u\ne
v}\varphi_{n-1}(sL_u),\quad n\in\mn.
\end{equation}
It is known and readily checked that $-\varphi_{n}'(s)$ is the Laplace transform of $W_n$ under the size-biased measure $\wh\mmp$. Let $V_{n}$ be a random variable with $\mmp(V_{n}\in\cdot)=\wh\mmp(W_{n}\in\cdot)$ (the use of $V_{n}$ is for our convenience and allows us to do all subsequent calculations under $\mmp$ only).
Then (\ref{ik}) is equivalent to the distributional identity
\begin{equation}\label{rec}
V_n\ \od\ MV_{n-1}+T_n,\quad n\in\mn,
\end{equation}
where $(M,T_n)$ is a random vector independent of $V_{n-1}$ and
with distribution
\begin{align}\label{defi2}
\begin{split}
\mmp\{(M, T_n)\in B\}\ &=\
\me\left(\sum_{|u|=1}L_{u}\1_B\Bigg(L_{u}, \sum_{v\ne
u}L_{v}W_{n-1}(v)\Bigg)\right)\\
&=\ \wh\mmp\left\{\Bigg(\Pi_{1},\sum_{v\in\mathcal{I}_{1}}
L_{v}W_{n-1}(v)\Bigg)\in B\right\},\quad n\in\mn,
\end{split}
\end{align}
where $B\subset \mr^2$ is any Borel set. An application of
Minkowski's inequality in $L_{p-1}$ yields
\begin{equation}\label{equ}
||V_n||_{p-1}\ \le\ ||M||_{p-1}||V_{n-1}||_{p-1}+||T_n||_{p-1}.
\end{equation}
Arguing in the same way as in the proof of Lemma 4.2 in
\cite{Liu2000} one finds that
\begin{equation}\label{inter1}
||T_n||_{p-1}\ \le\ ||W_{n-1}||_{p-1} (\me W_1^p)^{1/(p-1)}.
\end{equation}
For the remaining discussion we distinguish between two cases:

\vspace{.2cm}\noindent
{\sc Case 1}. $m(p)=1$

\vspace{.1cm}\noindent Note that
$||M||^{p-1}_{p-1}=\wh\me\Pi_{1}^{p-1}=m(p)=1$ and that $m(p-1)<1$
(by log-convexity of $m$). Hence $\sup_{n\ge
0}||W_n||_{p-1}<\infty$ by Proposition \ref{lp}, and we obtain
from (\ref{equ}) that $||V_n||_{p-1}=O(n)$ or, equivalently,
$$\me W_n^p\ =\ \me V_n^{p-1}\ =\ O(n^{p-1}). $$

\vspace{.2cm}\noindent
{\sc Case 2}. $m(p)>1$

\vspace{.1cm}\noindent Assume first that $p\in(2,3]$. Then we
conclude from (\ref{inter1}) and the already established part of
the lemma that $||T_n||_{p-1}=O(m^{n/(p-1)}(p))$, regardless of
the (finite) value of $m(p-1)$. By (\ref{defi2}),
$||M||_{p-1}=m^{1/(p-1)}(p)$ whence we conclude from (\ref{equ})
that $||V_n||_{p-1}=O(nm^{n/(p-1)}(p))$ or, equivalently, that
$$ \me W_n^p\ =\ \me V_n^{p-1}\ =\ O(n^{p-1}m^{n}(p)).$$
The subsequent proof proceeds by induction over $b$. Suppose that
we have already verified that $\me W_n^p=O(n^{b(p-1)}m^{n}(p))$
when $p\in (b+1,b+2]$. In order for proving $\me
W_n^p=O(n^{(b+1)(p-1)}m^{n}(p))$ when $p\in (b+2,b+3]$ it suffices
to note that $||T_n||_{p-1}=O(n^b m^{n/(p-1)}(p))$ and that any
solution to the recursive inequality
$$ c_{n}\ \le\ dc_{n-1}+O(n^{b} d^{n}),\quad n\in\mn,\ c_0=1, $$
with $d>1$ satisfies $c_n=O(n^{b+1}d^n)$. This completes the proof.
\end{proof}

We mention in passing that the distributional identity \eqref{rec}, obtained above with the help of Laplace transforms (see \eqref{ik}), may also be derived by a probabilistic
argument using the results stated in Section 2 on size biasing and spinal trees. However,
we refrain from supplying further details.

As we will make multiple use of the following version of
Burkholder's inequality (see \cite[Theorem 1 on p.~396]{Chow}), it is
stated here for ease of~reference.
\begin{lemma}\label{bu}
Let $p>1$ and $\{Z_n: n\in\mn\}$ be a martingale with $Z_{0}=0$ and a.s.\ limit $A$. Then $\me |Z|^p< \infty$ if and only if $\me
\left(\sum_{n\ge 0} (Z_{n+1}-Z_n)^2\right)^{p/2}<\infty$.
If one of these holds then
$$ c_p \left\|\left(\sum_{n\ge 0}
(Z_{n+1}-Z_n)^2\right)^{1/2}\right\|_p\ \le\ \|Z\|_p\ \le\ C_p
\left\|\left(\sum_{n\ge 0} (Z_{n+1}-Z_n)^2\right)^{1/2}\right\|_p, $$
where $c_p:=(p-1)/(18p^{3/2})$ and $C_p:=18p^{3/2}/(p-1)^{1/2}$.
\end{lemma}

%To simplify reading we decided to collect all the notation here,
%at the end of Introduction.

\section{Proofs of Theorem \ref{main1} and \ref{main2}}

Before proceeding with the proof of the main results, put $\mu_{p}:=\me|W_{1}-1|^{p}$, $R:=\sum_{n\ge 0}e^{2an}(W_{n+1}-W_{n})^{2}$ and recall that $W_{n}(v)$ denotes the copy of $W_{n}$ pertaining to the subtree $v\mathbf{V}$ rooted at $v$. Then
\begin{equation}\label{mincr}
W_{n+1}-W_{n}\ =\ \sum_{|v|=n}L_{v}(W_{1}(v)-1)
\end{equation}
for $n\in\mn_{0}$. Let us also stipulate hereafter that $C\in (0,\infty)$ denotes a generic constant which may differ from line to line.

\begin{proof}[Proof of Theorem \ref{main1}]
In view of Burkholder's inequality (Lemma \ref{bu}) it is clear that $L_{p}$-convergence of $\wh{A}=\sum_{n\ge 0}e^{an}(W_{n+1}-W_{n})$ holds true if, and only if, $R$ exists in $L_{p/2}$. Suppose the latter be true and recall that $\gamma=\inf\{m^{1/r}(r):r\in [1,2]\}$. Then, by a double use of Jensen's inequality in combination with Burkholder's inequality,
\begin{align*}
\me R^{p/2}\ &=\ \me\left(\sum_{n\ge 0}\Bigg(\sum_{|v|=n}e^{an}L_{v}(W_{1}(v)-1)\Bigg)^{2}\right)^{p/2}\\ %\noalign{\break}
&\ge\ N^{p/2}\,\me\left(\frac{1}{N}\sum_{n=0}^{N-1}\Bigg(\sum_{|v|=n}e^{an}L_{v}(W_{1}(v)-1)\Bigg)^{2}\right)^{p/2}\\ %\noalign{\break}
&\ge\ \frac{1}{N^{1-p/2}}\sum_{n=0}^{N-1}\me\Bigg|\sum_{|v|=n}e^{an}L_{v}(W_{1}(v)-1)\Bigg|^{p}\\ %\noalign{\break}
&\ge\ \frac{C}{N^{1-p/2}}\sum_{n=0}^{N-1}\me\Bigg(\sum_{|v|=n}e^{2an}L_{v}^{2}(W_{1}(v)-1)^{2}\Bigg)^{p/2}\\ %\noalign{\break}
&=\ \frac{1}{N^{1-p/2}}\sum_{n=0}^{N-1}e^{pan}\,\me\left(Z_n^{(2)}
\sum_{|v|=n}\frac{L_{v}^{2}}{Z_n^{(2)}}(W_{1}(v)-1)^2\right)^{p/2}\\
&\ge\ \frac{C\mu_{p}}{N^{1-p/2}}\sum_{n=0}^{N-1}e^{pan}\,\me\big(Z_n^{(2)}\big)^{p/2}\\
&=\ \frac{C\mu_{p}}{N^{1-p/2}}\sum_{n=0}^{N-1}e^{pan}s_{n}(2).
\end{align*}
This proves necessity of $\mu_{p}<\infty$ and
$\lim_{n\to\infty}e^{pa}s^{1/n}_{n}(2)\le 1$. But the last limit
equals either $e^{pa}\gamma^{p}$ (if $p\le\vth$) or $e^{pa}m(p)$
by an appeal to Proposition \ref{limitchar}. Therefore we have
proved necessity of $e^{a}\inf_{r\in [p,2]}m^{1/r}(r)\le 1$ as
claimed.

\vspace{.3cm}\noindent
Now suppose that $\mu_{r}<\infty$ and $e^{a}m^{1/r}(r)<1$ for some $r\in [p,2]$. By an appeal to Burkholder's inequality in combination with the subadditivity of $x\mapsto x^{p/2}$ and $x\mapsto x^{r/2}$, we infer
\begin{align*}
\me R^{p/2}\ &\le\ \sum_{n\ge 0}e^{pan}\,\me|W_{n+1}-W_{n}|^{p}\\
&\le\ C\sum_{n\ge 0}e^{pan}\,\me\Bigg(\sum_{|v|=n}L_{v}^{2}(W_{1}(v)-1)^{2}\Bigg)^{p/2}\\
&\le\ C\sum_{n\ge 0}e^{pan}\,\me\Bigg(\sum_{|v|=n}L_{v}^{r}|W_{1}(v)-1|^{r}\Bigg)^{p/r}.
\end{align*}
Use Jensen's inequality to see that
$$ \me\Bigg(\sum_{|v|=n}L_{v}^{r}|W_{1}(v)-1|^{r}\Bigg)^{p/r}\ \le\ \mu_{r}^{p/r}\,m^{p/r}(r) $$
and thus
$$ \me R^{p/2}\ \le\ C\mu_{r}^{p/r}\sum_{n\ge 0}e^{pan}m^{p/r}(r)\ <\ \infty. $$
This completes the proof of Theorem \ref{main1}.
\end{proof}

\begin{rem}\label{sharper}\rm
If $p>\vth$, i.e.\ $m^{1/\vth}(\vth)<m^{1/p}(p)=\min_{r\in [p,2]}m^{1/r}(r)$, the proof of Proposition \ref{limitchar} (see \eqref{29}) has actually shown that
$\nu m^{n}(p)\le s_{n}(2)\le m^{n}(p)$ for all $n\in\mn_{0}$ and some $\nu\in (0,1)$. In this case we hence obtain
\begin{align*}
\me R^{p/2}\ &\ge\ \frac{\mu_{p}}{N^{1-p/2}}\sum_{n=0}^{N-1}e^{pan}s_{n}(2)
\ \ge\ \frac{\nu\mu_{p}}{N^{1-p/2}}\sum_{n=0}^{N-1}e^{pan}m^{n}(p)
\end{align*}
and thereby conclude that the $L_{p}$-convergence of $\wh{A}$ or,
equivalently, $\me R^{p/2}<\infty$ can only hold true if
$e^{a}m^{1/p}(p)<1$. In the case where the function $r\mapsto
m^{1/r}(r)$ attains its minimum at some $\theta\ge 2$, we arrive
at a similar conclusion, because then
$\lim_{n\to\infty}m^{-np/2}(2)s_{n}(2)$ exists and is positive by
\eqref{30}. This confirms our assertions stated in Remark
\ref{supplement}.
\end{rem}

\begin{proof}[Proof of Theorem \ref{main2}]
We show first necessity of condition (\ref{1}) and thus assume
that the series $A$ in (\ref{series}) converges in $L_p$. By Lemma \ref{asa} and Remark \ref{rem2.3}(a), the same holds true for $\wh{A}$ and by an appeal to Lemma \ref{bu}
\begin{equation}\label{de}
\me R^{p/2}<\infty
\end{equation}
As $p\ge 2$, the function $x\mapsto x^{p/2}$ is superadditive and thus
\begin{equation}\label{3}
\sum_{n\ge 0} e^{pan}\,\me |W_{n+1}-W_n|^p\ \le\ \me
R^{p/2}<\infty.
\end{equation}
It is clear from \eqref{mincr} that $W_{n+1}-W_n$ is
the a.s.\ limit of a martingale (see, for example, \cite[Section
3]{AlsKuhl} for more details). Consequently, by another appeal to Lemma \ref{bu} and the afore-mentioned superadditivity,
\begin{align*}
\me |W_{n+1}-W_n|^p\ &\ge\ C\, \me\left(\sum_{|v|=n}L_v^2(W_1(v)-1)^2\right)^{p/2}\\
&\ge\ C\,\me\left(\sum_{|v|=n}L_v^p |W_1(v)-1|^p\right)\\
&=\ C\mu_{p}m^{n}(p)
\end{align*}
for $n\in\mn_{0}$ This inequality together with (\ref{3}) implies the necessity of $\me W_1^p<\infty$ and $e^{ap}m(p)<1$ for the $L_p$-convergence of $A$. Moreover, by Jensen's inequality,
\begin{align*}
\me\left(\sum_{|v|=n}L_v^2(W_1(v)-1)^2\right)^{p/2}\ &\ge
\ \left(\me\left(\sum_{|v|=n}L_v^2(W_1(v)-1)^2\right)\right)^{p/2}\\
&\ =\ \mu_{2}^{p/2}m(2)^{n}
\end{align*}
which together with (\ref{3}) finally gives the asserted necessity of
$e^{a}m^{1/2}(2)<1$.

\vspace{.2cm}\noindent
Let us now turn to the sufficiency of conditions (\ref{1}). By Lemma \ref{asa} and Remark \ref{rem2.3}(a) it suffices to verify $L_{p}$-convergence of $\wh{A}$. By combining \eqref{mincr}, another use of Lemma \ref{bu}, the convexity of $x\mapsto x^{p/2}$ and a conditioning with respect to $\mathcal{F}_n$, we infer
\begin{align*}
\me |W_{n+1}-W_n|^p\ &\le\ C\,\me\Bigg(\sum_{|v|=n}L_v^2(W_1(v)-1)^2\Bigg)^{p/2}\\
&\le\ C\,\me\Bigg(Z_n^{(2)}\sum_{|v|=n}\frac{L_v^2}{Z_n^{(2)}}(W_1(v)-1)^2\Bigg)^{p/2}\\
&\le\ C\mu_{p}\,\me\big(Z_n^{(2)}\big)^{p/2}
\end{align*}
whence it is enough to verify
\begin{equation}\label{inter}
e^{pan}\,\me\big(Z_n^{(2)}\big)^{p/2}\ =\ O(q^n)\quad
\text{for some }q\in (0,1),
%A:=\sum_{n\ge 0} b_n^p \me
%\left(\sum_{|v|=n}L_v^2\right)^{p/2}<\infty.
\end{equation}
hereafter. Indeed, it then follows that
\begin{equation}\label{im}
e^{2an}\me^{2/p}|W_{n+1}-W_{n}|^{p}\ =\ O(q^{2n/p}).
\end{equation}
and thereby with the help of Minkowski's inequality in $L_{p/2}$ and once more
Lemma \ref{bu}
$$ \me|\wh{A}|^{p}\ \le\ C\,\me R^{p/2}\le\ C\left(\sum_{n\ge 0}
e^{2an}\me^{2/p}|W_{n+1}-W_n|^{p})\right)^{p/2}\ <\ \infty. $$

So let us prove (\ref{inter}) for the case $p>2$, for it trivially
holds with $q=e^{2a}m(2)$ in the case $p=2$. %Recall that Recall
%that
%$$ \me\big(W_{1}^{(2)}\big)^{r}\ =\ \me\left(\sum_{|v|=1}\Bigg(\frac{L_v^2}{m(2)}\Bigg)^r\right)\ =
%\ \frac{m(2r)}{m^r(2)}$$ and
Notice that $\me\big(W_{1}^{(2)}\big)^{p/2}\le m^{-p/2}(2)\me
W_1^p <\infty$. For the remaining discussion we distinguish two
cases:

\vspace{.2cm}\noindent {\sc Case 1}.  $m(p)<m^{p/2}(2)$.

\vspace{.2cm}\noindent Then the second condition in (\ref{1})
reads $e^{a}m^{1/2}(2)<1$. By Proposition \ref{lp} applied to
$W_{n}^{(2)}$ and $p/2$ instead of $W_n$ and $p$, we obtain
$\sup_{n\ge 0}\me\big(W_{n}^{(2)}\big)^{p/2}<\infty$ and  this
ensures validity of (\ref{inter}) with $q=e^{ap}m^{p/2}(2)$.

\vspace{.2cm}\noindent {\sc Case 2}.  $m(p)\ge m^{p/2}(2)$.

\vspace{.2cm}\noindent
Then the second condition in (\ref{1}) takes the form $e^{a}m^{1/p}(p)$ $<1$. Lemma \ref{ma} applied to $W_{n}^{(2)}$ and $p/2$ instead of $W_n$ and $p$ provides us with $\me\big(W_{n}^{(2)}\big)^{p/2}=O(n^c m^{n}(p)m(2)^{-pn/2})$ for $c>0$. Consequently, $\me\big(Z_{n}^{(2)}\big)^{p/2}=O(n^c m^{n}(p))$ which proves validity of (\ref{inter}) with  $q=\delta e^{pa}m(p)$ for
some $\delta>1$ sufficiently close to $1$. The proof is herewith
complete.
\end{proof}

\end{document}